\documentclass{article}

\usepackage{amsmath,amsthm,amssymb,amsfonts}
\usepackage{verbatim}
\usepackage{graphicx}

\usepackage{hyperref}

\usepackage{setspace} 
\usepackage{xcolor}

\usepackage{tikz}

\usepackage{tkz-berge,tkz-graph}

\usepackage{longtable}

\theoremstyle{plain}
\newtheorem{thm}{Theorem}

\newtheorem{conj}[thm]{Conjecture}

\theoremstyle{definition}

\numberwithin{thm}{section}

\DeclareMathOperator{\id}{id}
\def\Z{{\mathbb Z}}

\begin{document}
\bibliographystyle{hplain}

\author{P. Christopher Staecker}
\title{Some enumerations of binary digital images}

\maketitle

\begin{abstract}
The topology of digital images has been studied much in recent years, but no attempt has been made to exhaustively catalog the structure of binary images of small numbers of points. We produce enumerations of several classes of digital images up to isomorphism and decide which among them are  homotopy equivalent to one another. Noting some patterns in the results, we make some conjectures about digital images which are irreducible but not rigid.
\end{abstract}

\tikzset{vertex/.style={circle,draw,fill,inner sep=0pt,minimum size=1mm}}

\section{Introduction}
A \emph{[binary] digital image}, or simply \emph{image}, is a pair $(X,\kappa)$ where $\kappa$ is a symmetric antireflexive binary relation on $X$, called the \emph{adjacency relation}. Typically the set $X$ is taken to be a finite subset of $\Z^n$ for some $n\ge 2$, and the relation $\kappa$ encodes some common notion of ``adjacency'' in $\Z^n$. Notable adjacency relations are the ``4-adjacency'' and ``8-adjacency'' in $\Z^2$. Two points of $\Z^2$ are 4-adjacent when their coordinates match in one position and differ by 1 in the other. Two different points are 8-adjacent when their coordinates differ by at most 1 in both coordinates. (The numbers 4 and 8 refer to the total number of points in $\Z^2$ which are adjacent to a given point using either relation.)

The topological theory of binary digital images has been studied by many authors for the past 30 years. The work has mostly been characterized by making discretized versions of definitions from classical topology, and seeing where the definitions lead. Continuity, connectivity, homotopy, and homology and fundamental groups have all been defined and studied by various authors. 

Though the definitions are motivated directly by the classical theory of topological spaces, digital images are fundamentally discrete in nature, and are usually taken to be finite. As in other discrete fields, it would be useful to have exhaustive catalogs of these images (at least the interesting ones) for small numbers of points, and this is the goal of the current paper.

We have produced the enumerations mostly by brute force computer search, and have carried the computations as far as can be done in reasonable time on an ordinary computer. Most of the computations and enumerations could be improved with heavier machinery, though all algorithms carried out have exponential complexity in the number of points.

We begin with a quick survey of some standard definitions which are repeated throughout the literature, see for example \cite{boxe05}. A digital image technically is a pair $(X,\kappa)$, but we will omit the $\kappa$ in almost all cases because the adjacency relation will be implicit. A \emph{digital interval} $[a,b]_\Z$ for $a,b\in \Z$ with $a\le b$ is the set $\{a,a+1,\dots b\}$, equipped with the natural ``2-adjacency'' wherein $x$ is adjacent only to $x+1$ and $x-1$. 

A function $f:X\to Y$ is \emph{continuous} if and only if whenever $a,b\in X$ are adjacent, the points $f(a),f(b) \in Y$ are equal or adjacent. Such a function is an \emph{isomorphism} when it has a continuous inverse. Equivalently, $f$ is an isomorphism when it gives a graph isomorphism on the \emph{adjacency graph of $X$}, the simple graph whose vertices are the points of $X$ and edges are given by the adjacency relation. 

For two continuous functions $f,g:X\to Y$, a \emph{homotopy} from $f$ to $g$ is a continuous function $H:X\times [0,k]_\Z \to Y$ where $k \in \Z$ and $H(x,0) = f(x)$ and $H(x,k) = g(x)$ for all $x$. For $H$, ``continuity'' means that $H(x,t)$ is continuous in $x$ for fixed $t$ as a map $X \to Y$, and also continuous in $t$ for fixed $x$ as a map $[0,k]_\Z \to Y$. If such a $H$ exists we say $f$ is \emph{homotopic} to $g$, and we write $f\simeq g$. If there is a homotopy from $f$ to $g$ having $k=1$, then we say $f$ and $g$ are \emph{homotopic in one step}. 

Two spaces $X$ and $Y$ are \emph{homotopy equivalent} when there are continuous maps $f:X \to Y$ and $g:Y\to X$ whose compositions $f\circ g$ and $g\circ f$ are each homotopic to identity maps. 

The paper \cite{hmps15} investigates homotopy equivalence in detail. A finite image $X$ is called \emph{reducible} if it is homotopy equivalent to an image of fewer points. Otherwise it is called \emph{irreducible}. Lemma 2.9 of \cite{hmps15} shows that an image $X$ is reducible if and only if the identity map on $X$ is homotopic to a nonsurjection in one step.

We say that an image $X$ is \emph{pointed reducible} if the identity map is homotopic in one step to a nonsurjection which fixes at least one point. Otherwise we say it is \emph{pointed irreducible}. 

Some irreducible images have the stronger property that no map at all is homotopic to the identity except for the identity itself. In \cite{hmps15} such spaces are called \emph{rigid}. Thus we have the following implications:
\[ X \text{ is rigid } \implies X \text{ is irreducible } \implies X \text{ is pointed irreducible } \]
None of the converses for the above implications are true. In \cite{hmps15} an example is given of an image which is pointed irreducible but not irreducible, and another example is given of an image which seems to be irreducible but is not rigid. Our enumerations will show that this latter example is in fact irreducible.

Since [pointed] irreducibility and rigidity of a disconnected image are equivalent to the same properties of each of its connected components, we consider only connected images (i.e. images whose adjacency graphs are connected).

The structure of this paper is as follows: In Section \ref{abstractsection} we describe the enumeration process for all ``abstract'' connected digital images (simple graphs, not necessarily realizable as subsets of $\Z^n$ with some natural geometric adjacency relation) up to isomorphism. In Section \ref{4adjsection} we describe an enumeration up to isomorphism of all connected images in $\Z^2$ with 4-adjacency having at most 12 points. In Section \ref{8adjsection} we describe an enumeration to isomorphism of all connected images in $\Z^2$ with 8-adjacency having at most 9 points. For all these enumerations, we determine which images are rigid, which are irreducible, and which are pointed irreducible. 

Throughout, we have used the open-source free software project Sage for the computations. Source code is available at the author's website for inspection and experimentation. Also available are graphics and machine-readable data files describing all digital images enumerated.\footnote{\url{http://faculty.fairfield.edu/cstaecker}}

\section{Abstract connected digital images}\label{abstractsection}
In this section we consider all connected digital images $(X,\kappa)$, without any regard for whether they naturally embed as subsets of $\Z^n$ using any standard adjacency relation. Up to isomorphism, such images are characterized by their adjacency graphs, which are connected simple graphs. 

Enumerations of all connected simple graphs are well-known, and this suffices to enumerate all connected abstract digital images up to isomorphism. To decide rigidity, irreducibility, and pointed irreducibility for these images it suffices to consider all functions homotopic to the identity in one step, and verify which among these are continuous and surjective. In the worst case there are $d^n-1$ functions (continuous or not) different from the identity in one step, where $d$ is the maximum degree of any vertex, and $n$ is the number of vertices. For small enough $n$ it is feasible to check each of these functions for continuity and surjectivity. We obtain the following counts:

\begin{center}
\begin{tabular}{r|rrrrrrrrr}
$n$  &	1  &	2  &	3  &	4  &	5  &	6  &	7 
&	8  &	9 \\
\hline
Images & 1& 1 & 2 & 6 & 21 & 112 & 853 & 11117 & 261080 \\
Pointed irreducible images &	1 &	0 &	0 &	0 &	1
&	2 &	9 &	68 &	1110 \\
Irreducible images &	1 &	0 &	0 &	0 &	1 &	1
&	3 &	28 &	547 \\
Rigid images &	1 &	0 &	0 &	0 &	0 &	0 &	2
&	26 &	544 \\
\end{tabular}
\end{center}

From the last two rows above we see that there is at most one irreducible image which is not rigid for $n< 7$. This image for each $n$ is the cycle of $n$ points, which we denote $C_n$. Boxer shows in \cite{boxe05} that $C_n$ is irreducible for $n\not\in\{2,3,4\}$, and it is clearly nonrigid because the rotation maps are continuous and homotopic to the identity. 

For $n=8$ we have one nonrigid irreducible image other than $C_8$, and this is the example image given in \cite[Example 3.14]{hmps15}. For $n=9$ there are 2 nonrigid irreducible images other than $C_9$. These three nonrigid irreducible images are shown in Figure \ref{nonrigidirr}, along with their graph6 string identifiers.\footnote{The graph6 format is an ASCII string representation for simple graphs developed as part of the nauty package \cite{nauty}. In most major computer algebra systems, graphs can be input simply by their graph6 strings.} In all cases the adjacency graphs are nonplanar. Planarity does not seem obviously related to rigidity, but to motivate a search for a counterexample we conjecture:
\begin{conj}\label{planar}
Let $X$ be a non-rigid irreducible digital image which is not isomorphic to $C_n$. Then the adjacency graph of $X$ is not planar.
\end{conj}

\begin{figure}
\begin{center}
\begin{tabular}{ccc}
\begin{tikzpicture}
\foreach \x in {0,1,2,3,4,5,6,7} { \node[vertex](v\x) at (360*\x/8:1.2) {}; };
\draw (0:1.2) \foreach \x in {1,2,3,4,5,6,7,8} { -- (360*\x/8 : 1.2) };
\foreach \x in {0,1,2,3} {\draw (360*\x/8:1.2) -- (360*\x/8 + 180:1.2);}
\end{tikzpicture} 
&
\begin{tikzpicture}
\foreach \x in {0,1,2,3,4,5,6,7,8} { \node[vertex](v\x) at (360*\x/9:1.2) {}; };
\draw (0:1.2) \foreach \x in {1,2,3,4,5,6,7,8,9} { -- (360*\x/9 : 1.2) };
\foreach \x in {0,1,2,3,4,5,6,7,8} {\draw (360*\x/9:1.2) -- (360*\x/9 + 80:1.2);}
\end{tikzpicture}
\qquad
&
\begin{tikzpicture}
\foreach \x in {0,1,2,3} { \node[vertex](a\x) at (360*\x/4 + 45:.6) {}; };
\foreach \x in {0,1,2,3} { \node[vertex](b\x) at (360*\x/4 + 45:1.2) {}; };
\draw (a1) -- (a2);
\draw (a3) -- (a0);
\draw (a1) to[out=-25, in=115] (a3) ;
\draw (a0) to[out=-155, in=65] (a2);
%\draw (b3) \foreach \x in {0,1,2,3} { -- (b\x)};
\draw (b3) -- (b0) -- (b1) -- (b2) -- (b3);
%\draw (0,-1.2) \foreach \x in {0,1,2,3} { -- (b\x)};
\foreach \x in {0,1,2,3} {\draw (a\x) -- (b\x); }
\node[vertex](c) at (0,0) {};
\draw (c) \foreach \x in {0,1,2,3} { -- (360*\x/4 + 45:.6)} ;
\end{tikzpicture}
\\
GrDKPK & HhciKeX & HzSW[Mb
\end{tabular}
\end{center}
\caption{The three nonrigid irreducible images on 9 points or fewer, with their graph6 strings.\label{nonrigidirr} In the first image (which has no vertex in the center) the ``antipodal'' map is homotopic to the identity. In the second a rotation is homotopic to the identity. In the third the vertical reflection, fixing the center point, is homotopic to the identity.}
\end{figure}
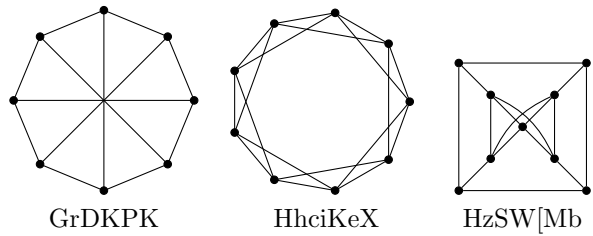

Among the irreducible images enumerated above, it is natural to ask which are homotopy equivalent to one another. Theorem 2.10 of \cite{hmps15} shows that irreducible images are homotopy equivalent if and only if they are isomorphic. Thus the enumerations above suffice to determine the homotopy equivalence classes of all images listed.

\section{Connected images in $\Z^2$ with 4-adjacency}\label{4adjsection}
In this section we consider finite digital images which can be expressed as subsets of $\Z^2$ with the 4-adjacency relation. That is, two points are adjacent when their coordinates match in one position, and differ by 1 in the other. Informally, points in $\Z^2$ are 4-adjacent when they are ``next to'' each other, not allowing diagonals.

Finite connected digital images with 4-adjacency naturally correspond to \emph{polyominoes}, well studied combinatorial objects consisting of finitely many square tiles in a connected grid arrangement. See \cite{golo96} for a standard reference. Distinct polyominoes can represent isomorphic images: for example there are, up to rotation and translation, two distinct polyominoes of 3 tiles (one straight and one bent), but they are isomorphic as digital images with 4-adjacency. 

Our enumeration of all digital images with 4-adjacency is based on the most common inductive algorithm for enumerating polyominoes. To generate all $n$-tile polyominoes, first we generate all $(n-1)$-tile polyominoes, and then test all ways in which one tile can be added to these. No simple formula is known for the number of polyominoes of $n$ tiles, but this number is known to grow exponentially in $n$.

We enumerate digital images with 4-adjacency by constructing polyominoes, and then separating these into isomorphism classes with respect to 4-adjacency by checking if the adjacency graphs are isomorphic. The polyomino enumeration has exponential complexity in $n$ (it must, since the number of polyominoes grows exponentially in $n$), and the graph isomorphism problem is also famously nonpolynomial, using the best current algorithms.

When we measure irreducibility of these images, there is some subtlety. A given image $X$ may be homotopy equivalent to some other image $Y$, but we would also like to know whether or not $Y$ is realizable as a subset of $\Z^2$ with 4-adjacency. In fact it always is, as the following simple theorem shows:
\begin{thm}
Let $(X,\kappa)$ be reducible. Then there is a proper subset $Y\subset X$ such that $(X,\kappa)$ is homotopy equivalent to $(Y,\kappa)$.
\end{thm}
\begin{proof}
The important content of the theorem is that the same adjacency relation is being used for both $X$ and $Y$ in the conclusion. Our proof is essentially the same as the proof of Lemma 2.8 of \cite{hmps15}. We provide the full argument for the sake of completeness.

Since $X$ is reducible, say $(X,\kappa)$ is homotopy equivalent to $(Z,\kappa')$, where $Z$ has fewer points than $X$. Then there are continuous maps $f:X \to Z$ and $g:Z \to X$ such that in particular $g\circ f \simeq \id_X$. But since $Z$ has fewer points than $X$ the map $g\circ f$ cannot be surjective. Thus $Y = g(f(X))$ is a proper subset of $X$. 

Now it is easy to show that $(X,\kappa)$ is homotopy equivalent to $(Y,\kappa)$. Let $h = g\circ f: X \to Y$, and let $i:Y \to X$ be the inclusion. Then we have $h \circ i: Y \to Y$ and $i \circ h: X \to X$. Both of these maps are continuous, and both are homotopic to identity maps since $h$ is homotopic to the identity. Thus $X$ is homotopy equivalent to $Y$.
\end{proof}

In particular the above shows that a digital image in $\Z^2$ with 4-adjacency is reducible if and only if it reduces to another digital image in $\Z^2$ with 4-adjacency. Thus, among the set of images with 4-adjacency, there is no distinction between ``reducible'' and ``reducible to an image with 4-adjacency''.

Below are the computational results of the enumerations. In all cases for these $n$, an image is pointed irreducible if and only if it is irreducible.

\begin{center}
\begin{tabular}{r|rrrrrrrrrrrr}
$n$  &	1  &	2  &	3  &	4  &	5  &	6  &	7 
&	8  &	9  &	10  &	11  &	12 \\
\hline
Images with 4-adjacency &	1 &	1 &	1 &	3 &	4 &	10
&	19 &	51 &	112 &	300 &	746 &	2042 \\
Pointed irreducible &	1 &	0 &	0 &	0 &	0 &	0
&	0 &	1 &	0 &	1 &	0 &	1 \\
Irreducible &	1 &	0 &	0 &	0 &	0 &	0 &	0
&	1 &	0 &	1 &	0 &	1 \\
Rigid &	1 &	0 &	0 &	0 &	0 &	0 &	0 &	0
&	0 &	0 &	0 &	0 \\
\end{tabular}
\end{center}

It seems that the equality of the middle two rows might hold in general, but this is not the case. Arguments similar to those in Section 5 of \cite{hmps15} show that the 13-point image in Figure \ref{4adjptfig} (a) is reducible but pointed irreducible.
\begin{figure}
\begin{center}
\begin{tabular}{cc}
\begin{tikzpicture}[scale=.2]
	\filldraw[fill=gray, xshift=2cm,yshift=2cm]
		(45:1.2) \foreach \x in {135,225,315,45} { -- (\x:1.2) };
	\filldraw[fill=gray, xshift=2cm,yshift=4cm]
		(45:1.2) \foreach \x in {135,225,315,45} { -- (\x:1.2) };
	\filldraw[fill=gray, xshift=2cm,yshift=6cm]
		(45:1.2) \foreach \x in {135,225,315,45} { -- (\x:1.2) };
	\filldraw[fill=gray, xshift=2cm,yshift=8cm]
		(45:1.2) \foreach \x in {135,225,315,45} { -- (\x:1.2) };
	\filldraw[fill=gray, xshift=4cm,yshift=2cm]
		(45:1.2) \foreach \x in {135,225,315,45} { -- (\x:1.2) };
	\filldraw[fill=gray, xshift=4cm,yshift=8cm]
		(45:1.2) \foreach \x in {135,225,315,45} { -- (\x:1.2) };
	\filldraw[fill=gray, xshift=6cm,yshift=2cm]
		(45:1.2) \foreach \x in {135,225,315,45} { -- (\x:1.2) };
	\filldraw[fill=gray, xshift=6cm,yshift=6cm]
		(45:1.2) \foreach \x in {135,225,315,45} { -- (\x:1.2) };
	\filldraw[fill=gray, xshift=6cm,yshift=8cm]
		(45:1.2) \foreach \x in {135,225,315,45} { -- (\x:1.2) };
	\filldraw[fill=gray, xshift=8cm,yshift=2cm]
		(45:1.2) \foreach \x in {135,225,315,45} { -- (\x:1.2) };
	\filldraw[fill=gray, xshift=8cm,yshift=4cm]
		(45:1.2) \foreach \x in {135,225,315,45} { -- (\x:1.2) };
	\filldraw[fill=gray, xshift=8cm,yshift=6cm]
		(45:1.2) \foreach \x in {135,225,315,45} { -- (\x:1.2) };
	\filldraw[fill=gray, xshift=8cm,yshift=8cm]
		(45:1.2) \foreach \x in {135,225,315,45} { -- (\x:1.2) };
\end{tikzpicture}
&
\begin{tikzpicture}[scale=.2]
	\filldraw[fill=gray, xshift=2cm,yshift=4cm]
		(45:1.2) \foreach \x in {135,225,315,45} { -- (\x:1.2) };
	\filldraw[fill=gray, xshift=2cm,yshift=6cm]
		(45:1.2) \foreach \x in {135,225,315,45} { -- (\x:1.2) };
	\filldraw[fill=gray, xshift=2cm,yshift=8cm]
		(45:1.2) \foreach \x in {135,225,315,45} { -- (\x:1.2) };
	\filldraw[fill=gray, xshift=4cm,yshift=2cm]
		(45:1.2) \foreach \x in {135,225,315,45} { -- (\x:1.2) };
	\filldraw[fill=gray, xshift=4cm,yshift=10cm]
		(45:1.2) \foreach \x in {135,225,315,45} { -- (\x:1.2) };
	\filldraw[fill=gray, xshift=6cm,yshift=2cm]
		(45:1.2) \foreach \x in {135,225,315,45} { -- (\x:1.2) };
	\filldraw[fill=gray, xshift=6cm,yshift=6cm]
		(45:1.2) \foreach \x in {135,225,315,45} { -- (\x:1.2) };
	\filldraw[fill=gray, xshift=6cm,yshift=10cm]
		(45:1.2) \foreach \x in {135,225,315,45} { -- (\x:1.2) };
	\filldraw[fill=gray, xshift=8cm,yshift=4cm]
		(45:1.2) \foreach \x in {135,225,315,45} { -- (\x:1.2) };
	\filldraw[fill=gray, xshift=8cm,yshift=8cm]
		(45:1.2) \foreach \x in {135,225,315,45} { -- (\x:1.2) };
	\filldraw[fill=gray, xshift=10cm,yshift=6cm]
		(45:1.2) \foreach \x in {135,225,315,45} { -- (\x:1.2) };
\end{tikzpicture}
\\ (a) & (b)
\end{tabular}
\end{center}
\caption{(a): An image in $\Z^2$ with $4$-adjacency which is reducible but pointed irreducible, and (b) a similar image in $\Z^2$ with $8$-adjacency.\label{4adjptfig}}
\end{figure}
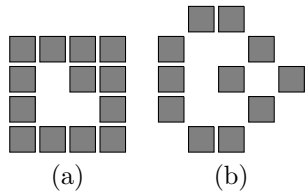

For even $n$ not equal to 2 or 6, there are images in $\Z^2$ with 4-adjacency which are isomorphic to $C_n$. When $n>4$, such images will be irreducible and not rigid, and so the last two rows should differ by at least 1 for those $n$. There seem to be no other nonrigid irreducible images with 4-adjacency. Indeed, if Conjecture \ref{planar} holds then no other such images can exist, since all images in $\Z^2$ with 4-adjacency have planar adjacency graph. Even if Conjecture \ref{planar} fails, however, it still seems reasonable to conjecture as follows:
\begin{conj}\label{4adjconj}
Let $X$ be a digital image in $\Z^2$ with 4-adjacency. Then $X$ is non-rigid and irreducible if and only if $X$ is isomorphic to $C_n$ for some $n>4$.
\end{conj}

\section{Connected images in $\Z^2$ with 8-adjacency}\label{8adjsection}
In this section we consider finite digital images which can be expressed as subsets of $\Z^2$ with 8-adjacency, in which two different points are adjacent when their coordinates differ by at most 1 in each position.

Enumeration of these images is analogous to the enumeration for $4$-adjacency. We begin with all polyomino-like objects, this time allowing diagonal adjacency. These combinatorial objects are called \emph{polyplets} (in \cite{golo96} they are called \emph{pseudo-polyominoes}), and are also well-studied. There is no known formula for the number of polyplets on $n$ tiles, but this number must grow exponentially since it is greater than the number of polyominoes. 

Recursive enumerations similar to those in the previous section for 4-adjacency produce the following counts:

\begin{center}
\begin{tabular}{r|rrrrrrrrr}
$n$  &	1  &	2  &	3  &	4  &	5  &	6  &	7 
&	8  &	9 \\
\hline
Images with 8-adjacency &	1 &	1 &	2 &	6 &	15
&	51 &	173 &	681 &	2682 \\
Pointed irreducible &	1 &	0 &	0 &	0 &	0 &	1
&	1 &	1 &	1 \\
Irreducible &	1 &	0 &	0 &	0 &	0 &	1 &	1
&	1 &	1 \\
Rigid &	1 &	0 &	0 &	0 &	0 &	0 &	0 &	0
&	0 \\
\end{tabular}
\end{center}

Again the equality of the middle two rows will only hold for small $n$. The 11-point image in Figure \ref{4adjptfig} (b) is pointed irreducible but not irreducible, according to the same arguments cited for the example of Figure \ref{4adjptfig} (a). This second image is discussed in detail in \cite{bs15}

Comparing the tables in this and the previous section we see that, for each $n$, there are more images on $n$ points with 8-adjacency than with 4-adjacency (which is hardly surprising). In fact each image with 4-adjacency is naturally isomorphic to an image with 8-adjacency. In the following, for a set $X\subset \Z^2$, the pair $(X,k)$ represents the digital image $X$ with k-adjacency.
\begin{thm}
Let $X \subset \Z^2$ be a digital image. Then there is some $Y\subset \Z^2$ such that $(X,4)$ is isomorphic to $(Y,8)$. 
\end{thm}
\begin{proof}
First we note that it is not true in general that we may take $Y=X$. For example if $X = \{(0,0),(1,0),(0,1),(1,1)\}$, then $(X,4)$ is not isomorphic to $(X,8)$ since the adjacency graph of $(X,8)$ is a complete graph while the adjacency graph of $(X,4)$ is not. 

Let $f:\Z^2 \to \Z^2$ be given by $f(x,y) = (x+y,x-y)$. The set $f(\Z^2)$ consists of pairs $(a,b)$ having the same parity. We claim that $(X,4)$ is isomorphic to $(f(X),8)$. It suffices to show that point $p,q\in X$ are 4-adjacent if and only if $f(p),f(q) \in f(X)$ are 8-adjacent. 4-adjacency of $p$ and $q$ means that their coordinates differ by one in one position, which means that the coordinates of $f(p)$ and $f(q)$ differ by one in both positions, and thus are 8-adjacent. Similarly the converse will hold, and so $f$ is an isomorphism as desired.
\end{proof}

The above theorem shows that the number of 8-images for some given $n$ is always greater or equal to the number of 4-images. In fact the number is strictly greater (typically much greater) whenever $n\ge3$  since there will always exist images with 8-adjacency having sets of 3 mutually adjacent points, while this is impossible using 4-adjacency.

As in the previous section, nonrigid irreducible images seem rare. We conjecture:
\begin{conj}
Let $X$ be a digital image in $\Z^2$ with 8-adjacency. Then $X$ is non-rigid and irreducible if and only if  $X$ is isomorphic to $C_n$ for some $n>4$.
\end{conj}

\end{document}